\documentclass[ 12pt,a4paper,openany]{article}

\usepackage{amsmath,amssymb,amsfonts,amsthm}
\usepackage{graphics}
\usepackage{mathrsfs}
\usepackage[all]{xy}

\newtheorem{thm}{Theorem}[section]
\newtheorem{prop}[thm]{Proposition}
\newtheorem{lemma}[thm]{Lemma}

\theoremstyle{definition}

\newtheorem{rem}[thm]{Remark}

\newtheorem*{ack}{Acknowledgment}
\newtheorem{condition}[thm]{Condition}

\newcommand{\sq}{\hfill $\square$}
\newcommand{\kah}{K\"{a}hler}
\newcommand{\dol}{\sqrt{-1}\partial \overline{\partial}}

\newcommand{\ma}{Monge-Amp$\grave{{\rm e}}$re}

\makeatletter
\def\address#1#2{\begingroup
\noindent\parbox[t]{16cm}{%
\small{\scshape\ignorespaces#1}\par\vskip1ex
\noindent\small{\itshape E-mail address}%
\/: #2\par\vskip4ex}\hfill%
\endgroup}%
\makeatother

\makeatletter

\@addtoreset{equation}{section}
\makeatother

\title{Complete scalar-flat K\"{a}hler metrics on affine algebraic manifolds}
\author{Takahiro Aoi}
\date{}
\begin{document}
\maketitle


%
%









\maketitle

\begin{abstract}
Let $(X,L_{X})$ be an $n$-dimensional polarized manifold.\
Let $D$ be a smooth hypersurface defined by a holomorphic section of $L_{X}$.\
We prove that if $D$ has a constant positive scalar curvature K\"{a}hler metric, $X \setminus D$ admits a complete scalar-flat K\"{a}hler metric, under the following three conditions: (i) $n \geq 6$ and there is no nonzero holomorphic vector field on $X$ vanishing on $D$, (ii) an average of a scalar curvature on $D$ denoted by $\hat{S}_{D}$ satisfies the inequality $0 < 3 \hat{S}_{D} < n(n-1)$, (iii) there are positive integers $l(>n),m$ such that the line bundle $K_{X}^{-l} \otimes L_{X}^{m}$ is very ample and the ratio $m/l$ is sufficiently small.\

\end{abstract}

\setcounter{tocdepth}{2}
\tableofcontents

\section{
Introduction}
\label{intro}
The existence of constant scalar curvature {\kah} (cscK) metrics on complex manifolds is a fundamental problem in {\kah} geometry.\
If a complex manifold is noncompact, there are many positive results in this problem.\
In 1979, Calabi \cite{Ca} showed that if a Fano manifold has a {\kah} Einstein metric, then there is a complete Ricci-flat {\kah} metric on the total space of the canonical line bundle.\
In addition, there exist following generalizations.\
In 1990, Bando-Kobayashi \cite{BK} showed that if a Fano manifold admits an anti-canonical smooth divisor which has a Ricci-positive {\kah} Einstein metric, then there exists a complete Ricci-flat {\kah} metric on the complement (see also \cite{TY}).\
Tian-Yau \cite{TY1} showed that if a Fano manifold admits an anti-canonical smooth divisor which has a Ricci-flat {\kah} metric, then there is a complete Ricci-flat {\kah} metric on the complement.\
In 2002, on the other hand, as a scalar curvature version of Calabi's result \cite{Ca}, Hwang-Singer \cite{HS} showed that if a polarized manifold has a nonnegative cscK metric, then the total space of the dual line bundle admits a complete scalar-flat {\kah} metric.\
However, a similar generalization of Hwang-Singer \cite{HS} like Bando-Kobayashi \cite{BK} and Tian-Yau \cite{TY1} is unknown since it is hard to solve a forth order nonlinear partial differential equation.\

In this paper, assuming the existence of a smooth hypersurface which admits a constant positive scalar curvature {\kah} metric, we will prove the existence of a complete scalar-flat {\kah} metric on the complement of this hypersurface by using the results in \cite{Aoi1} and \cite{Aoi2}.\
Our proof goes roughly as follows.\\
{\it Step} 1. We show that if the smooth hypersurface has a cscK metric, there is a complete {\kah} metric whose scalar curvature decays at a higher order.\\
{\it Step} 2. We show that the existence of a complete {\kah} metric whose scalar curvature is sufficiently small implies the existence of a complete scalar-flat {\kah} metric.\\
{\it Step} 3. We construct a complete {\kah} metric on the complement of the smooth hypersurface, whose scalar curvature is arbitrarily small.\\
{\it Step} 4. Finally, we show the existence of a complete scalar-flat {\kah} metric by solving the forth order nonlinear partial differential equation.\

Now we describe our strategy which contains the results in the previous papers \cite{Aoi1} and \cite{Aoi2} more precisely.\
Let $(X,L_{X})$ be a polarized manifold of dimension $n$, i.e., $X$ is an $n$-dimensional compact complex manifold and $L_{X}$ is an ample line bundle over $X$.\
Assume that there is a smooth hypersurface $D \subset X$ with
$$
D \in |L_{X}|.
$$
Set an ample line bundle $L_{D} := \mathscr{O}(D)|_{D} = L_{X}|_{D}$ over $D$.\
Since $L_{X}$ is ample, there exists a Hermitian metric $h_{X}$ on $L_{X}$ which defines a {\kah} metric $\theta_{X}$ on $X$, i.e., the curvature form of $h_{X}$ multiplied by $\sqrt{-1}$ is positive definite.\
Then, the restriction of $h_{X}$ to $L_{D}$ defines also a {\kah} metric $\theta_{D}$ on $D$.\ 
Let $\hat{S}_{D}$ be the average of the scalar curvature $S(\theta_{D})$ of $\theta_{D}$ defined by
$$
\hat{S}_{D} := \frac{\displaystyle\int_{D} S(\theta_{D}) \theta_{D}^{n-1}}{\displaystyle\int_{D} \theta_{D}^{n-1}} = \frac{(n-1) c_{1}(K_{D}^{-1}) \cup c_{1}(L_{D})^{n-2}}{c_{1}(L_{D})^{n-1}},
$$
where $K_{D}^{-1}$ is the anti-canonical line bundle of $D$.\
Note that $\hat{S}_{D}$ is a topological invariant in the sense that it is representable in terms of Chern classes of the line bundles $K_{D}^{-1}$ and $L_{D}$.\
In this paper, we treat the following case :
\begin{equation}
\label{positivity 1}
\hat{S}_{D} > 0.
\end{equation}
Let $\sigma_{D} \in H^{0}(X,L_{X})$ be a defining section of $D$ and set $t := \log || \sigma_{D} ||_{h_{X}}^{-2}$.\
Following \cite{BK}, we can define a complete {\kah} metric $\omega_{0}$ by
$$
\omega_{0}:= \frac{n(n-1)}{\hat{S}_{D}}\dol \exp \left(\frac{\hat{S}_{D}}{n(n-1)}t\right)\\
$$
on the noncompact complex manifold $X \setminus D$.\
In addition, since $(X \setminus D, \omega_{0})$ is of asymptotically conical geometry (see \cite{BK} or Section 4 of \cite{Aoi1}), we can define weighted Banach spaces $C^{k,\alpha}_{\delta} = C^{k,\alpha}_{\delta}(X \setminus D)$ for $k \in \mathbb{Z}_{\geq 0}, \alpha \in (0,1)$ and with a weight $\delta \in \mathbb{R}$ with respect to the distance function $r$ defined by $\omega_{0}$ from some fixed point in $X \setminus D$.\
It follows from the construction of $\omega_{0}$ that $S(\omega_{0}) = O(r^{-2})$ near $D$.\

\medskip
{\it Step} 1. The cscK condition implies the following stronger decay property (see \cite{Aoi1}).\

\begin{thm}
\label{scalar curvature decay}
If $\theta_{D}$ is a constant positive scalar curvature {\kah} metric on $D$, i.e., $S(\theta_{D}) = \hat{S}_{D} > 0,$
we have
$$
S(\omega_{0}) = O(r^{-2 -2n(n-1)/\hat{S}_{D}})
$$
as $r \to \infty$.
\end{thm}
Thus, the cscK condition implies that $S(\omega_{0}) \in C^{k,\alpha}_{\delta}$ for some $\delta > 2$ and  any $k,\alpha$.\

\medskip
{\it Step} 2. To construct a complete scalar-flat {\kah} metric on $X \setminus D$, the linearization of the scalar curvature operator plays an important role :
$$
L_{\omega_{0}} = - \mathcal{D}^{*}_{\omega_{0}}\mathcal{D}_{\omega_{0}} + (\nabla^{1,0} \ast , \nabla^{0,1} S(\omega_{0}))_{\omega_{0}}.
$$
Here, $\mathcal{D}_{\omega_{0}} = \overline{\partial} \circ \nabla^{1,0}$.\
We will show that if $4 < \delta <2n $ and there is no nonzero holomorphic vector field on $X$ which vanishes on $D$, then $\mathcal{D}^{*}_{\omega_{0}}\mathcal{D}_{\omega_{0}} : C^{4,\alpha}_{\delta - 4} \to C^{0,\alpha}_{\delta}$ is isomorphic.\
For such operators, we consider the following :

\begin{condition}
\label{condition A}
Assume that $n \geq 3$ and there is no nonzero holomorphic vector field on $X$ which vanishes on $D$.\
For $4 < \delta <2n $, the operator
$$
L_{\omega_{0}} : C^{4,\alpha}_{\delta - 4} \to C^{0,\alpha}_{\delta}
$$
is isomorphic, i.e., we can find a constant $\hat{K}>0$ such that
$$
||L_{\omega_{0}} \phi ||_{C^{0,\alpha}_{\delta}} \geq \hat{K} || \phi ||_{C^{4,\alpha}_{\delta - 4}}
$$
for any $ \phi \in C^{4,\alpha}_{\delta - 4} $.\
\end{condition}

In addition, we consider
\begin{condition}
\label{condition B}
$$
|| S(\omega_{0}) ||_{C^{0,\alpha}_{\delta}} < c_{0}\hat{K}/2.
$$
\end{condition}

Here, the constant $c_{0}$ is defined Lemma 6.2 in \cite{Aoi1}.\
Under these conditions, we have the following result (see \cite{Aoi1}) :

\begin{thm}
\label{complete scalar-flat}
Assume that $n \geq 3$ and there is no nonzero holomorphic vector field on $X$ which vanishes on $D$.\
Assume that $\theta_{D}$ is a constant scalar curvature {\kah} metric satisfying
$$
0 < \hat{S}_{D} < n(n-1).
$$
Assume moreover that Condition $\ref{condition A}$ and Condition $\ref{condition B}$ hold,
then $X \setminus D$ admits a complete scalar-flat {\kah} metric.\
\end{thm}

In fact, we can show the existence of a complete scalar-flat {\kah} metric on $X \setminus D$ under the following assumptions : (i) $n \geq 3$ and there is no nonzero holomorphic vector field on $X$ which vanishes on $D$, (ii) there exists a complete {\kah} metric on $X \setminus D$ which is of asymptotically conical geometry, such that its scalar curvature is sufficiently small and decays at a higher order.\
So, if there exists a complete {\kah} metric on $X \setminus D$ which is sufficiently close to $\omega_{0}$ at infinity, satisfying Condition \ref{condition A} and Condition \ref{condition B}, we can show the existence of a complete scalar-flat {\kah} metric on $X \setminus D$.\

Theorem \ref{complete scalar-flat} is proved by the fixed point theorem on the weighted Banach space $C^{4,\alpha}_{\delta - 4}(X \setminus D)$ by following Arezzo-Pacard \cite{AP1}, \cite{AP2} (see also \cite{Sz1}).\
In general, constants $c_{0}, \hat{K}$ which arise in Condition \ref{condition A} and Condition \ref{condition B} depend on the background {\kah} metric $\omega_{0}$.\
In addition, to construct such a {\kah} metric, we have to find a complete {\kah} metric $X \setminus D$ whose scalar curvature is arbitrarily small.\

\medskip
{\it Step} 3. We consider a degenerate (meromorphic) complex {\ma} equation.\
Take positive integers $l > n$ and $m$ such that the line bundle $K_{X}^{-l} \otimes L_{X}^{m}$ is very ample.\
Let $F \in |K_{X}^{-l} \otimes L_{X}^{m}|$ be a smooth hypersurface defined by a holomorphic section $\sigma_{F} \in H^{0}(X,K_{X}^{-l} \otimes L_{X}^{m})$ such that the divisor $D+F$ is simple normal crossing.\
For a defining section $\sigma_{D} \in H^{0}(X,L_{X})$ of $D$, set
$$
\xi := \sigma_{F} \otimes \sigma_{D}^{-m}.
$$
From the result due to Yau \cite[Theorem 7]{Yau}, we can solve the following degenerate complex Monge-Amp$\grave{{\rm e}}$re equation:
$$
(\theta_{X} + \dol \varphi)^{n} = \xi^{-1/l} \wedge \overline{\xi}^{-1/l}.
$$
Moreover, it follows from a priori estimate due to  Ko\l odziej \cite{Ko} that the solution $\varphi$ is bounded on $X$.\
Thus, we can glue plurisubharmonic functions by using the regularized maximum function.\
To compute the scalar curvature of the glued {\kah} metric, we need to study behaviors of higher order derivatives of the solution $\varphi$.\
So, we give explicit estimates of them near the intersection $D \cap F$ (see \cite{Aoi2}) :
\begin{thm}
\label{solution}
Let $(z^{i})_{i=1}^{n} = (z^{1},z^{2},...,z^{n-2},w_{F},w_{D})$ be local holomorphic coordinates such that $\{ w_{F} = 0 \} = F$ and $\{ w_{D} = 0 \} = D$.\
Then, there exists a positive integer $a(n)$ depending only on the dimension $n$ such that 

\begin{eqnarray*}
\left| \frac{\partial^{2}}{\partial z^{i} \partial \overline{z}^{j}} \partial^{\alpha} \varphi \right| &=& O\left( |w_{D}|^{-2a(n) m/l}|w_{F}|^{-2a(n) / l}  \right),\\
\left| \frac{\partial^{4}}{\partial w_{F}^{2} \partial \overline{w_{F}^{2}}}  \varphi \right| &=& O\left( |w_{D}|^{-2 a(n) m/l}|w_{F}|^{ -2 - 2a(n)/l } \right),\\
\left| \frac{\partial^{4}}{\partial w_{D}^{2} \partial \overline{w_{D}^{2}}} \varphi \right| &=& O\left( |w_{D}|^{ -2 -2a(n)m/l }|w_{F}|^{-2a(n)/l} \right),
\end{eqnarray*}
as $|w_{F}|,|w_{D}| \to 0$, for any $1\leq i, j \leq n-2$ and multi-index $\alpha = (\alpha_{1},...,\alpha_{n})$ with $0 \leq \sum_{i} \alpha_{i} \leq 2$.\

\end{thm}

By applying Theorem \ref{solution} and gluing plurisubharmonic functions, we have the following result (see \cite{Aoi2}) :
\begin{thm}
\label{small scalar curvature}
Assume that there exist positive integers $l>n$ and $m$ such that
\begin{equation}
\label{positivity 2}
\frac{a(n)m}{2l} < \frac{\hat{S}_{D}}{n(n-1)}
\end{equation}
and the line bundle $K_{X}^{-l} \otimes L_{X}^{m}$ is very ample.\
Here, $a(n)$ is the positive integer in Theorem \ref{solution}.\
Take a smooth hypersuface $F \in |K_{X}^{-l} \otimes L_{X}^{m}|$ such that $D + F$ is simple normal crossing.\
Then, for any relatively compact domain $Y \subset \subset X \setminus (D \cup F)$, there exists a complete {\kah} metric $\omega_{F}$ on $X \setminus D$ whose scalar curvature $S(\omega_{F}) = 0$ on $Y$ and is arbitrarily small on the complement of $Y$.\
In addition, $\omega_{F} = \omega_{0}$ on some neighborhood of $D \setminus (D \cap F)$.\
\end{thm}

For example, if the anti-canonical line bundle $K_{X}^{-1}$ of the compact complex manifold $X$ is nef (in particular, $X$ is Fano), the assumption (\ref{positivity 2}) in Theorem \ref{small scalar curvature} holds, i.e., we can always find such integers $l,m$.\
In this paper, we treat the case that $K_{X}^{-1}$ has positivity in the senses of (\ref{positivity 1}) and (\ref{positivity 2}).\
From Theorem \ref{complete scalar-flat}, if there exists a complete {\kah} metric which is of asymptotically conical geometry and satisfies Condition \ref{condition A} and Condition \ref{condition B}, $X \setminus D$ admits a complete scalar-flat {\kah} metric.\
In fact, Theorem \ref{small scalar curvature} gives a {\kah} metric whose scalar curvature is under control.\
However, the {\kah} metric $\omega_{F}$ in Theorem \ref{small scalar curvature} is not of asymptotically conical geometry (near the intersection of $D$ and $F$).\
So, when we replace the complete {\kah} metric $\omega_{0}$ with $\omega_{F}$ obtained in Theorem \ref{small scalar curvature}, we can not apply Theorem \ref{complete scalar-flat} to a construction of a complete scalar-flat {\kah} metric.\

To solve this problem, we consider an average on some closed subset in $|K_{X}^{-l} \otimes L_{X}^{m+\beta}|$.\
Then, the asymptotically conicalness is recovered and we obtain the first result in this paper :
\begin{thm}
\label{asc and eq}
Assume that there are positive integers $l>n$ and $m$ such that the line bundle $K_{X}^{-l} \otimes L_{X}^{m}$ is very ample and
$$
\frac{a(n) m}{2l} < \frac{\hat{S}_{D}}{n(n-1)}.
$$
Then, there exists a complete {\kah} metric $\overline{\omega}$ on $X \setminus D$ satisfies following properties:
\begin{itemize}
\item $\overline{\omega}$ is equivalent to $\omega_{0}$, i.e., there is a constant $C>0$ such that
$$
C^{-1} \omega_{0} < \overline{\omega} < C \omega_{0}.
$$
Moreover, the {\kah} metric $\overline{\omega}$ is of asymptotically conical geometry.\
\item Assume that $n \geq 4$.\
If $\theta_{D}$ is cscK and $0<\hat{S}_{D}<n(n-1)$, the $C^{k,\alpha}$-norm of the scalar curvature $S(\overline{\omega})$ of weight $\delta \in (4 ,\min \{ 2n, 2 + 2n(n-1)/\hat{S}_{D} \})$ can be made arbitrarily small.\
\end{itemize} 
\end{thm}

Thus, we obtain the {\kah} metric $\overline{\omega}$ which is of asymptotically conical geometry.\
In addition, the scalar curvature of $\overline{\omega}$ is arbitrarily small in the sense of the weight norm.\

\medskip
{\it Step} 4. Finally, by applying the similar argument in the proof of Theorem \ref{complete scalar-flat} to the {\kah} metric $\overline{\omega}$ obtained in Theorem \ref{asc and eq}, we obtain our main result in this paper :

\begin{thm}
\label{complete scalar-flat kah}
Assume following conditions:
\begin{itemize}
\item $n\geq 6$ and there is no nonzero holomorphic vector field on $X$ which vanishes on $D$.\
\item The following inequality holds:
$$
0 < 3 \hat{S}_{D} < n(n-1).
$$
\item There are positive integers $l>n$ and $m$ such that the line bundle $K_{X}^{-l} \otimes L_{X}^{m}$ is very ample and
$$
\frac{a(n) m}{2l} < \frac{\hat{S}_{D}}{n(n-1)}.
$$
\end{itemize}
Then, if $D$ admits a cscK metric $\theta_{D}$, $X \setminus D$ admits a complete scalar-flat {\kah} metric.\
\end{thm}
In other word, we can solve the following forth order nonlinear partial differential equation:
$$
S(\overline{\omega} + \dol \phi) = 0, \hspace{7pt} \overline{\omega} + \dol \phi > 0, \hspace{7pt} \phi \in C^{4,\alpha}_{\delta - 4}
$$
for a weight $8 < \delta <\min \{ 2n, 2 + 2n(n-1)/\hat{S}_{D} \}$.\
The reason why we assume that $n \geq 6$ and $0 < 3 \hat{S}_{D} < n(n-1)$ in Theorem \ref{complete scalar-flat kah} is that we need the isomorphic Laplacian $\Delta_{\overline{\omega}}$ between higher order weighted Banach spaces.\

This paper is organized as follows.\
In Section 2, we will prove Theorem \ref{asc and eq}.\
Namely, we recover the asymptotically conicalness by constructing an average metric.\
In Section 3, we prove Theorem \ref{complete scalar-flat kah}, i.e., we show the existence of a complete scalar-flat {\kah} metric.\

\begin{ack}
The author would like to thank Professor Ryoichi Kobayashi who first brought the problem in this paper to his attention, for many helpful comments.\
\end{ack}

\section{
Proof of Theorem \ref{asc and eq}
}\label{sec:10}

In this section, we prove Theorem \ref{asc and eq}.\
We construct the complete {\kah} metric $\omega_{F} = \omega_{c,v,\eta}$ whose scalar curvature is arbitrarily small on $X \setminus D$ in \cite{Aoi2}.\
So, see the definitions of the parameters $c,v,\eta,\kappa$ and the functions $\Theta(t), G_v^\beta(\beta b), \tilde{G}_v^\beta(b)$ in Section 2 of \cite{Aoi2}.\
For $\beta \in \mathbb{Z}_{>0}$, take a holomorphic section $\sigma_{0}\in H^{0}(K_{X}^{-l} \otimes L_{X}^{m + \beta})$.\
We may assume that $D + F_{0}$ is simple normal crossing, where $F_{0}$ is a smooth hypersurface defined by $\sigma_{0}$.\
Let $(\sigma_{i})_{i} \subset H^{0}(K_{X}^{-l} \otimes L_{X}^{m+\beta})$ be an orthonormal basis with respect to the $L^2$ inner product.\
Take a sufficiently small number $\tau \in \mathbb{R}$.\
Write $h(l,m,\beta) := \dim H^{0}(K_{X}^{-l} \otimes L_{X}^{m + \beta})$.\
For $s = (s_{i})_{i} \in \mathbb{D}^{h(l,m,\beta)} := \{ z = (z_{i}) \in \mathbb{C}^{h(l,m,\beta)}| |z_{i}| \leq 1 \}$, define a meromorphic section of the line bundle $K_{X}^{-l} \otimes L_{X}^{m}$ by
$$
\sigma_{s} := (\sigma_{0} + \tau \sum_{i = 1}^{h(l,m,\beta)} s_{i} \sigma_{i}) \otimes \sigma_{D}^{-\beta}.
$$
Note that by taking a sufficiently small $\tau$, we may assume that $\sigma_{s} \neq 0$ for any $s \in \mathbb{D}^{h(l,m,\beta)}$.\
In addition,$\sigma_{s} \to \sigma_{0} \otimes \sigma_{D}^{-\beta}$ for any $s \in \mathbb{D}^{h(l,m,\beta)}$ as $\tau \to 0$.\
Let $F_{s}$ be a smooth hypersurface defined by ${\rm div}\sigma_{s} = F_{s} - \beta D$.\
Since $\sigma_{0}$ contained in $\sigma_{s}$ is not multiplied by $\tau$, the variation of $\tau$ affects the choice of $F_s$ if $\tau \neq 0$.\

By applying Theorem \ref{small scalar curvature}, we obtain a complete {\kah} metric $\omega_{F_{s}}$ with small scalar curvature for a meromorphic section $\sigma_{s} \otimes \sigma_{D}^{-m}$ of $K_{X}^{-l}$ (see \cite{Aoi2}).\
In fact, for a smooth function on $X \setminus (D \cup F_{s})$ defined by $b_{s} := \log || \sigma_{s} ||^{-2}$, we can obtain a {\kah} metric $\dol G_{v}^{\beta}(\beta b_{s})$ on $X$.\
Directly, we have
$$
\dol G_{v}^{\beta}(\beta b_{s}) = \left( \frac{1}{e^{- \beta b_{s}} + v} \right)^{1/\beta} \left( \beta \dol b_{s} + \frac{ e^{-\beta b_{s}}}{e^{-\beta b_{s}} + v} \sqrt{-1} \partial b_{s} \wedge \overline{\partial} b_{s} \right).
$$
This metric does not grow near $D$.\
When we glued plurisubharmonic functions in \cite{Aoi2}, we considered the {\kah} potential $\kappa \Theta(t) + G_{v}^{\beta}(\beta b_{s})$.\
In addition, $\lim_{b_{s} \to - \infty}\dol G_{v}^{\beta}(\beta b_{s}) > -\infty$.\
So, we can construct a complete {\kah} metric $\omega_{F_{s}}$ with small scalar curvature by using the regularized maximum $M_{\eta}$ (see \cite{De} or \cite{Aoi2}) to glue three plurisubharmonic functions $\Theta(t), \tilde{G}_{v}^{\beta}(b), t + \varphi + c$.\

$(X \setminus D, \omega_{F_{s}})$ is not of asymptotically conical geometry for any $s \in \mathbb{D}^{h(l,m,\beta)}$ (see Remark 4.2 in \cite{Aoi2}).\
To solve this problem, consider an average metric $\overline{\omega}$ defined by
$$
\overline{\omega} = \overline{\omega}(c,v,\eta,\tau) := \int_{\mathbb{D}^{h(l,m,\beta)}} \omega_{F_{s}} d\mu (s).
$$
Here, $\mu$ is the Lebesgue probability measure on $\mathbb{D}^{h(l,m,\beta)}$ and $c,v,\eta$ are parameters in the definition of $\omega_{F}$ in Theorem \ref{small scalar curvature} (see \cite{Aoi2}).\
Recall that $\eta = (\eta_{1},\eta_{2},\eta_{3})$ and $\eta_{1},\eta_{2} = O(c), \eta_{3}=O(1)$.\
To prove that $(X \setminus D, \overline{\omega})$ is of asymptotically conical geometry, it is enough to prove the following lemma:
\begin{lemma}
\label{eq}
For the {\kah} metric $\overline{\omega}$ defined above, we have
$$
\overline{\omega} - \omega_{0} = O( || \sigma_{D} ||^{2 \beta} )
$$
as $\sigma_{D} \to 0$.\
\end{lemma}

\begin{proof}
The region where $(X \setminus D,\omega_{F_{s}})$ is not of asymptotically conical geometry is defined by
\begin{equation}
\label{inequality 1}
| \Theta(t) - \tilde{G}_{v}^{\beta}(b_s) | < \eta_1 + \eta_2
\end{equation}
(see \cite{Aoi2}).\
For sufficiently large $b_{s}>0$, we have $v^{-1/\beta} \beta b_{s} \approx G_{v}^{\beta}(\beta b_{s})$.\
Here, $b_{s} := \log || \sigma_{s} ||^{-2}$.\
From the following inequality
$$
v^{-1/\beta} \beta b_{s} \approx G_{v}^{\beta}(\beta b_{s}) > (1 - \kappa) \Theta(t) - \eta_{1} - \eta_{2}
$$
obtained by (\ref{inequality 1}), we have
$$
|| \sigma_{s} ||^{2} < \exp \left( - (v^{1/\beta} /\beta )( (1 - \kappa) || \sigma_{D} ||^{-2\hat{S}_{D}/n(n-1)}  - \eta_{1} - \eta_{2} )\right).
$$
Take a point $p \in X \setminus D$ near $D$.\
Assume that $\sigma_{\tilde{s}}(p)=0$ for $\tilde{s} \in \mathbb{D}^{h(l,m,\beta)}$.\
Then, an element $s \in \mathbb{D}^{h(l,m,\beta)}$ satisfying the inequality above has to satisfy

\begin{eqnarray}
\label{region}
&&\nonumber || \tau \sum_{i = 1}^{h(l,m,\beta)} ( s_{i} - \tilde{s}_{i}) \sigma_{i}(p) \otimes \sigma_{D}(p)^{-\beta} ||^{2}\\
&<& \exp \left( - (v^{1/\beta} /\beta) ( (1 - \kappa) || \sigma_{D} ||^{-2\hat{S}_{D}/n(n-1)}  - \eta_{1} - \eta_{2} )\right).
\end{eqnarray}

By considering a suitable unitary transformation $u =(u_{i,j}) \in U(h(l,m,\beta))$, we can write as $\sum_{i = 1}^{h(l,m,\beta)} ( s_{i} - \tilde{s}_{i}) \sigma_{i}(p) = (\sum_{i,j = 1}^{h(l,m,\beta)} u_{i,j} ( s_{i} - \tilde{s}_{i})) \tilde{\sigma}(p)$ for some holomorphic section $\tilde{\sigma} \in H^{0}(K_{X}^{-l} \otimes L_{X}^{m + \beta})$ such that $\tilde{\sigma}(p) \neq 0$ and have
\begin{equation}
\label{region}
\left| \sum_{i,j = 1}^{h(l,m,\beta)} u_{i,j} ( s_{i} - \tilde{s}_{i}) \right|^{2} < \frac{ \exp \left( -( v^{1/\beta} /\beta) ( (1 - \kappa) || \sigma_{D} ||^{-2\hat{S}_{D}/n(n-1)}  - \eta_{1} - \eta_{2} )\right)}{ \tau^2|| \tilde{\sigma} \otimes \sigma_{D}^{-\beta}(p) ||^{2}}.
\end{equation}
Then, we have the following estimate
$$
\int \partial \Theta(t) \wedge \overline{\partial} \Theta(t) d\mu (s) = \exp \left( O\left(- || \sigma_{D} ||^{-2\hat{S}_{D}/n(n-1)} \right)\right).
$$
Next, we consider the term
$$
\int \partial G_{v}^{\beta}(\beta b_{s}) \wedge \overline{\partial} G_{v}^{\beta}(\beta b_{s}) d\mu (s)
$$
which appears in $\overline{\omega}$.\
From the inequality (\ref{inequality 1}), we have
$$
v^{-1/\beta} \beta b_{s} \approx G_{v}^{\beta}(\beta b_{s}) < (1 - \kappa) \Theta(t) + \eta_{1} + \eta_{2}.\
$$
Thus, the following inequality
$$
|| \sigma_{s} ||^{-2} < \exp \left( (v^{1/\beta} /\beta) ( (1 - \kappa) || \sigma_{D} ||^{-2\hat{S}_{D}/n(n-1)}  + \eta_{1} + \eta_{2} )\right)
$$
holds.\
Thus, we can estimate as follows
\begin{eqnarray}\nonumber
&&\int \partial G_{v}^{\beta}(\beta b_{s}) \wedge \overline{\partial} G_{v}^{\beta}(\beta b_{s}) d\mu (s)\\
&\leq& \exp \left( 2 v^{1/\beta} ( \eta_{1} +  \eta_{2} )/\beta \right)/\tau^2|| \tilde{\sigma} \otimes \sigma_{D}^{-\beta} ||^2. \label{gg}
\end{eqnarray}
By the definition of $\overline{\omega}$, we obtain
$$
\overline{\omega} \approx \left(1 - \exp \left( - || \sigma_{D} ||^{-2\hat{S}_{D}/n(n-1)} \right) \right) \omega_{0} + O(|| \sigma_{D} ||^{2 \beta} )
$$
near $D$.\
\end{proof}

{\it Proof of Theorem \ref{asc and eq}.}
Lemma \ref{eq} implies that the complete {\kah} manifold $(X \setminus D,\overline{\omega})$ is of asymptotically conical geometry.\
Thus, we will prove that the scalar curvature can be made small arbitrarily.\
To show  this, we take parameters $c,v,\tau$ and an integer $\beta$ so that
\begin{equation}
v^{1/\beta}c = k \log c,\hspace{5pt} \tau^2 = v, \hspace{5pt} \beta > \delta
\end{equation}
for a sufficiently large $k \in \mathbb{N}$ specified later.\

Firstly, from the construction of $\overline{\omega}$, weight norms of $S(\overline{\omega})$ away from $D\cup F_{0}$ can be made small arbitrarily by taking sufficiently small $\tau$.\
To show this, we study a function $f :\tau \to \overline{\omega}^{n}/\omega_{F_{0}}^{n}.$\
Note that this function is smooth and $f(0)=1$ and $\omega_{F_{0}}$ is Ricci-flat away from $D \cup F_{0}$.\
So, we have $S(\overline{\omega}) = O(\tau)$ away from $D \cup F_{0}$.\

Secondly, we study $S(\overline{\omega})$ near $F_{0}$ and away from $D$.\
We can write as
\begin{eqnarray*}
\overline{\omega}
&=& \int \left( \dol M_{c,v,\eta} \right) d \mu,
\end{eqnarray*}
where
\begin{eqnarray*}
M_{c,v,\eta}
&=& \frac{\partial M_{c,v,\eta}}{\partial t_{2}}  ( \gamma_{v}^{\beta} + \kappa \omega_{0}) + \frac{\partial M_{c,v,\eta}}{\partial t_{3}} \dol (t + \varphi)\\
&\hspace{15pt}+& \left[ 
\begin{array}{cc}
\partial \tilde{G}_{v}^{\beta}(b_{s})& \partial (t + \varphi)
\end{array} 
\right]
\left[ 
\begin{array}{c}
\frac{\partial^{2} M_{c,v,\eta}}{\partial t_{i} \partial t_{j}}
\end{array} 
\right]
\left[ 
\begin{array}{cc}
\overline{\partial} \tilde{G}_{v}^{\beta}(b_{s}) & \overline{\partial} (t + \varphi)
\end{array} 
\right]^{t}.
\end{eqnarray*}
On this region, we consider a sufficiently small neighborhood of $F_0$ by taking a sufficiently large parameter $c$.\
So, it is enough to consider the region defined by the following inequality
$$
G_{v}^{\beta}(b_{s}) + \kappa \Theta(t) - \eta_{2} > \max \{ \Theta(t) + \eta_{1} , t + \varphi + c + \eta_{3} \}.
$$
In addition, since we are considering the region away from $D$, by taking a sufficiently large parameter $c$, the inequality above can be rewritten as follows :
\begin{equation}
\label{near F}
G_{v}^{\beta}(b_{s}) + \kappa \Theta(t) - ( t + \varphi + c)  > \eta_{2} + \eta_{3}.
\end{equation}
So, we have
$$
|\sum_{i}^{h(l,m,\beta)} s_{i}|^2 < \exp \left( -(v^{1/\beta}/\beta) (t + \varphi + c - \kappa \Theta(t) + \eta_{2} + \eta_{3}) \right)/\tau^2 = O(c^{\beta-1}).
$$
Recall the definition of $\gamma_{v}^{\beta}$ in \cite{Aoi2} and the relation of the parameters $c,v$ :
$$
c v^{1/\beta} = k \log c
$$
So, the inequality above (\ref{near F}) implies that
$$
|| \sigma_{F_s} ||^{2} \leq v^{ k/ \beta}.
$$
Thus, we don't have to consider the case that $S(\omega_{F}) = O(1)$ and we have
\begin{equation}
\label{equivalence}
\overline{\omega} = \int \gamma_{v}^{\beta} d\mu(s) + \kappa \omega_{0} \approx v^{-1/\beta} \dol b + \kappa \omega_{0}.
\end{equation}
by taking a sufficiently large $c$.\
Since the Ricci form of $\dol (G_{v}^{\beta}(\beta b_{0}) + \kappa \Theta(t) )$ is bounded near $F_{0}$ and away from $D$, we can conclude that $S(\overline{\omega}) =O( v^{1/\beta} )$.\

Thirdly, we study $S(\overline{\omega})$ near $D$.\
Write
$$
\overline{\omega} = \omega_{0} + \dol \psi.
$$
By taking the trace with respect to the background metric $\omega_{0}$, we have
$$
\Delta_{\omega_{0}} \psi = {\rm tr}_{\omega_{0}} \overline{\omega} - n.
$$
To estimate
$$
S(\overline{\omega}) = S(\omega_{0} + \dol \psi),
$$
we study the right hand side in the equation above.\
Recall the construction of the complete {\kah} metric $\omega_{F_{s}}$.\
The bounded region where plurisubharmonic functions $\Theta(t), t + \varphi + c$ are glued is defined by following inequalities:
\begin{eqnarray*}
\tilde{G}_{v}^{\beta}(b_{s})  + \eta_{2} &<& \max \{\Theta(t) - \eta_{1}, ( t + \varphi +c) - \eta_{3} \},\\
| \Theta(t) - ( t + \varphi +c) | &<& \eta_{1} + \eta_{3}.
\end{eqnarray*}
In addition, $\omega_{F_{s}}$ is written as
\begin{eqnarray*}
\omega_{F_{0}}
&=& \frac{\partial M_{c,v,\eta}}{\partial t_{1}} \omega_{0} + \frac{\partial M_{c,v,\eta}}{\partial t_{3}} \dol (t + \varphi)\\
&\hspace{15pt}+& \left[ 
\begin{array}{cc}
\partial \Theta(t) & \partial (t + \varphi)
\end{array} 
\right]
\left[ 
\begin{array}{c}
\frac{\partial^{2} M_{c,v,\eta}}{\partial t_{i} \partial t_{j}}
\end{array} 
\right]
\left[ 
\begin{array}{cc}
\overline{\partial} \Theta(t) & \overline{\partial} (t + \varphi)
\end{array} 
\right]^{t}.
\end{eqnarray*}
Recall that $\eta_{1} + \eta_{3} = O(c)$.\
So, the inequality
\begin{eqnarray*}
| \Theta(t) - ( t + \varphi + c) | &<& \eta_{1} + \eta_{3}
\end{eqnarray*}
implies the following equivalence between complete {\kah} metrics:
$$
\frac{\partial M_{c,v,\eta}}{\partial t_{1}} \omega_{0} < \omega_{F_{s}} < 2 \omega_{0}
$$
on the region above.\
Next, we consider the region contained in the other region defined by
$$
v^{-1/\beta} \beta b_{s} \approx G_{v}^{\beta}(\beta b_{s}) > (1 - \kappa) \Theta(t) - \eta_{1} - \eta_{2}
$$
In order to estimate ${\rm tr}_{\omega_{0}} \overline{\omega} - n$, it is enough to estimate the following terms
$$
c^{-1} \int \partial \Theta(t) \wedge \overline{\partial} \Theta(t) d\mu (s) , \hspace{7pt} c^{-1} \int \partial G_{v}^{\beta}(\beta b_{s}) \wedge \overline{\partial} G_{v}^{\beta}(\beta b_{s}) d\mu (s).
$$
Since $c \leq \Theta(t)$ on this region, the first term can be estimated as follows
\begin{eqnarray*}
&& c^{-1} || \sigma_{D} ||^{-4\hat{S}_{D}/n(n-1)} \exp \left( - v^{1/\beta} (1 - \kappa) || \sigma_{D} ||^{-2\hat{S}_{D}/n(n-1)} /\beta \right)/ \tau^2||\tilde{\sigma} \otimes \sigma_{D}^{-\beta} ||^{2}\\
&=& O( c^{ 1 -  (1 - \kappa)k/\beta - n(n-1)\beta/\hat{S}_{D} +  \beta } (\log c)^{\beta})
\end{eqnarray*}
for parameters $\tau^2 = v, c v^{1/\beta} = k \log c$.\
From the estimate (\ref{gg}), the second term can be estimated as follows
\begin{eqnarray*}
&&c^{-1} \int \partial G_{v}^{\beta}(\beta b_{s}) \wedge \overline{\partial} G_{v}^{\beta}(\beta b_{s}) d\mu (s)\\
&\leq& \exp \left( 2 v^{1/\beta} ( \eta_{1} +  \eta_{2} )/\beta \right)/\tau^2|| \tilde{\sigma} \otimes \sigma_{D}^{-\beta} ||^2\\
&\leq& O(c^{-1 + 2 (a_{1} + a_{2})k/\beta + \beta - n(n-1)\beta/\hat{S}_{D} } )
\end{eqnarray*}
Recall the relation between parameters (see Claim 2 in Section 4 of \cite{Aoi2}) :
$$
1 - \kappa + \kappa a_1 - a_2 = 0.
$$
Since the choice of $a_{i} \in (0,1)$ is independent of $\beta,k$, we can choose sufficiently small $a_{i}$ and $\kappa$ which is sufficiently close to 1.\
Thus, we can make the following terms :
\begin{eqnarray*}
&& 1 -  (1 - \kappa)k/\beta - n(n-1)\beta/\hat{S}_{D} +  \beta\\
&& -1 + 2 (a_{1} + a_{2})k/\beta + \beta - n(n-1)\beta/\hat{S}_{D}
\end{eqnarray*}
negative by taking sufficiently large $\beta$ and $k$.\
Thus, we can estimate $\Delta_{\omega_{0}} \psi = {\rm tr}_{\omega_{0}} \overline{\omega} - n$ near $D$.\
From the equivalence (\ref{equivalence}), we obtain the following estimate near $F_0$ and away from $D$ :
$$
{\rm tr}_{\omega_{0}} \overline{\omega} - n = O(v^{-1/\beta}).
$$
For any weight $\epsilon \in (4, 2n)$, we have the following inequality
$$
\Delta_{\omega_{0}} \tilde{C} \rho^{-\epsilon+2} < - C \rho^{-\epsilon} < \Delta_{\omega_{0}} v^{1/\beta} \psi < C \rho^{-\epsilon} < - \Delta_{\omega_{0}} \tilde{C} \rho^{-\epsilon+2}
$$
on $X \setminus D$ for some constants $C,\tilde{C}>0$ depending only on $\epsilon$ and $n$.\
Here $\rho = || \sigma_{D} ||^{- \hat{S}_{D}/n(n-1)}$ is the barrier function defined in Section 5 of \cite{Aoi1} (see \cite{BK}).\
Thus, the maximum principle tells us that there is the following $C^{0}_{\epsilon - 4}$-estimate of $\psi$ :
\begin{equation}
\label{psi e}
|| \psi ||_{C^{0}_{\epsilon - 4}} \leq C v^{-1/\beta}.
\end{equation}
Recall the linearization of the scalar curvature operator
\begin{equation}
\label{linearization}
S(\overline{\omega}) = S(\omega_{0}) + L_{\omega_{0}} (\psi) + Q_{\omega_{0}} (\psi).
\end{equation}
In addition, the term $Q_{\omega_{0}} (\psi)$ can be written as
$$
Q_{\omega_{0}} (\psi) = (L_{\omega_{0} + s \dol \psi} - L_{\omega_{0}})(\psi)
$$
for some $s \in [0,1]$ (see \cite{Sz1} or \cite{Aoi1}).\
Choose $$\epsilon > \delta + 2.$$\
In this case, we can consider that $c \leq \Theta(t) \approx r^{2}$.\

Recall the interior Schauder estimate :
$$
|| \psi ||_{C^{4,\alpha}_{\epsilon - 4}} \leq C(\omega_{0}) ( || {\rm tr}_{\omega_{0}} \overline{\omega} - n ||_{C^{2,\alpha}_{\epsilon - 2}} + || \psi ||_{C^{0}_{\epsilon - 4}}  ).
$$
Here, $C(\omega_{0})$ is a positive constant depending only on $\omega_{0}$.\
The previous estimate (\ref{psi e}) implies that $|| \psi ||_{C^{4,\alpha}_{\epsilon - 4}} = O(v^{-1/\beta})$.\
Then, the equality (\ref{linearization}) implies that the norm of scalar curvature of {\it weight $\delta$} is estimated from above by $c^{(\delta - \epsilon)/2 + 1 } ( \log c)^{-1}$.\
In these settings of parameters, we show finally that the scalar curvature on the region defined by
$$
t + \varphi + c - \eta_{3} > \max \{ \Theta(t) + \eta_{1} , \tilde{G}_{v}^{\beta}(b_{s}) + \eta_{2} \}
$$
can be estimated in the sense of weighted norms.\
It follows from the first discussion that $S(\overline{\omega}) = O(\tau)$ on the region above.\
Then, we have
$$
S(\overline{\omega}) = O\left( \tau \right) = O\left( c^{-\beta/2} (k \log c )^{\beta/2} \right).
$$
Recall that $\Theta(t) \approx r^{2} \approx c$ on this region.\
So, we can estimate the $C^{k,\alpha}$-norm of the function $S(\overline{\omega})(r^2 + 1)^{\delta/2} \approx S(\overline{\omega})c^{\delta/2}$ in the definition of the weighted norm (see \cite{Aoi1}) on this region.\
More precisely, the choice of $\beta$ :
$$
\beta > \delta
$$
implies that we can estimate the $\delta$-weighted norm of the scalar curvature $S(\overline{\omega})$ on the region above.\
Therefore, from the discussion above, we can conclude that the weight norm of $S(\omega_{c,v,\eta})$ can be made small arbitrarily by taking a sufficiently large parameter $c$ (equivalently, sufficiently small parameters $v,\tau$).\
In addition, from the linearization of scalar curvatures, the scalar curvature $S(\overline{\omega})$ decays just like $S(\omega_{0})$.\
Thus, we finish the proof of Theorem \ref{asc and eq}.\
\sq

\begin{rem}
If $\theta_{D}$ is cscK, Theorem \ref{scalar curvature decay} implies that we have
$$
S(\overline{\omega}) = O(|| \sigma_{D} ||^{2+2\hat{S}_{D}/n(n-1)}) = O(r^{-2-2n(n-1)/\hat{S}_{D}})
$$
near $D$ (see \cite{Aoi1}).\
\end{rem}

\begin{rem}
Recall that we choose a parameter $v>0$ so that the inequality $(|| \sigma_{F} ||^{2\beta} + v)^{2/\beta} < || \sigma_{F}||^{4am/l}$ holds on the region defined by
\begin{eqnarray*}
 \Theta(t) + \eta_{1} &<& \max \{\tilde{G}_{v}^{\beta}(b) - \eta_{2}, ( t + \varphi +c) - \eta_{3} \},\\
| \tilde{G}_{v}^{\beta}(b) - ( t + \varphi +c) | &<& \eta_{2} + \eta_{3}
\end{eqnarray*}
in \cite{Aoi2}.\
Note that $G_{v}^{\beta}(\beta b) \approx \beta v^{-1/\beta} b$ for sufficiently large $b > 0$.\
The choice of parameters $cv^{1/\beta}= k \log c$ in the previous theorem implies that we have $||\sigma_{F}||^{-2\beta} \approx v^{k}$.\
Therefore, we can choose a suitable parameter $v> 0$ so that $(|| \sigma_{F} ||^{2\beta} + v)^{2/\beta} < || \sigma_{F}||^{4am/l}$ without contradiction (see Remark 2.5 in \cite{Aoi2}).\
\end{rem}

\section{
Proof of Theorem \ref{complete scalar-flat kah}}
\label{sec:11}

After this, all weighted Banach spaces $C^{k,\alpha}_{\delta} = C^{k,\alpha}_{\delta}(X\setminus D)$ are defined by the fixed {\kah} metric $\omega_{0}$.\
In Theorem \ref{complete scalar-flat kah}, we assume that
$$
0 < 3\hat{S}_{D} < n(n-1)
$$
and we choose a weight $\delta$ so that
\begin{equation}
\label{weight}
8 < \delta < \min \{ 2n, 2 + 2n(n-1)/\hat{S}_{D} \}
\end{equation}
and a function
$$
\phi \mathcal{D}_{\overline{\omega}}^{*}\mathcal{D}_{\overline{\omega}} \phi
$$
is integrable for $\phi \in C^{4,\alpha}_{\delta - 4}$ with respect to the volume form $\overline{\omega}^{n}$.\
In addition, we may assume that the integer $a(n)$ in Theorem \ref{small scalar curvature} satisfies
$$
12/a(n) <  \delta - 8 < \min \{ 2n - 8, 2n(n-1)/\hat{S}_{D} - 6 \}.
$$

\subsection{Condition \ref{condition A} and Condition \ref{condition B}}

In this subsection, we show that Condition \ref{condition A} and Condition \ref{condition B} (see Introduction of this paper or \cite{Aoi1}) hold with respect to the complete {\kah} metric $\overline{\omega}$ obtained in the previous section.\

In order to find the constant $\hat{K}$ in Condition \ref{condition A}, we use the resonance theorem (see \cite[p.69]{Yosida}).\
\begin{thm}[the resonance theorem]
\label{resonance}
Let $\{ T_{a} \hspace{3pt} | \hspace{3pt} a \in A \}$ be a family of bounded linear operators defined on a Banach space $\mathcal{X}$ into a normed linear space $\mathcal{Y}$.\
Then, the boundedness of $\{ || T_{a} x || \hspace{3pt} | \hspace{3pt} a \in A \}$ for each $x \in \mathcal{X}$ implies the boundedness of $\{ || T_{a} || \hspace{3pt} | \hspace{3pt} a \in A \}$.\
\end{thm}

Then, we obtain the following theorem which is the core of this paper :

\begin{thm}
\label{key theorem}
Take parameters $c,v,\tau$ so that $v^{1/\beta} c = k \log c, \tau^2 = v$.\
Assume that there is no nonzero holomorphic vector field on $X$ which vanishes on $D$.\
Then, there exists an uniform constant $K>0$ such that
$$
|| \mathcal{D}_{\overline{\omega}}^{*}\mathcal{D}_{\overline{\omega}} \phi ||_{C^{0,\alpha}_{\delta}} \geq K || \phi ||_{C^{4,\alpha}_{\delta - 4}}
$$
for any $c,v,\eta,\tau$ and $\phi \in C^{4,\alpha}_{\delta - 4}$.
\end{thm}

\begin{proof}
We prove this theorem by using Theorem \ref{resonance}.\
So, for a fixed function $\phi \in C^{4,\alpha}_{\delta - 4}$, it is enough to show that the quantity
$$
\frac{|| \phi ||_{C^{4,\alpha}_{\delta - 4}}}{|| \mathcal{D}_{\overline{\omega}}^{*}\mathcal{D}_{\overline{\omega}} \phi ||_{C^{0,\alpha}_{\delta}}}
$$
has an upper bound depending only on $\phi$.\
We prove this by contradiction.\
Assume that there exists a sequence $(\tau, v , c) \to (0, 0, \infty)$ such that $|| \mathcal{D}_{\overline{\omega}}^{*}\mathcal{D}_{\overline{\omega}} \phi ||_{C^{0,\alpha}_{\delta}} \to 0$ for some $ \phi \in C^{4,\alpha}_{\delta - 4}$ with $|| \phi ||_{C^{4,\alpha}_{\delta - 4}} = 1$.\
By integration by parts, we have
\begin{eqnarray*}
\int_{X \setminus D} \phi \mathcal{D}_{\overline{\omega}}^{*}\mathcal{D}_{\overline{\omega}} \phi \overline{\omega}^{n}
&=& \int_{X \setminus D} | \mathcal{D}_{\overline{\omega}} \phi |^{2} \overline{\omega}^{n}.
\end{eqnarray*}
Recall that $\mathcal{D}_{\overline{\omega}} \to \mathcal{D}_{\dol(t + \varphi)}$ as $(\tau, v , c) \to (0, 0 . \infty)$.\
We show that
$$
\int_{X \setminus D} \phi \mathcal{D}_{\overline{\omega}}^{*}\mathcal{D}_{\overline{\omega}} \phi \overline{\omega}^{n} \to 0
$$
as $(\tau, v , c) \to (0, 0 , \infty)$.\
To see this, we study the volume of the subset $\cup_{s \in \mathbb{D}^{h(l,m+1)}} F_{s}$.\
For $p \in X \setminus D$ close to $F_{0}$, we can find $s \in \mathbb{D}^{h(l,m+1)}$ such that $\sigma_{s}(p)=0$.\
So, we have
$$
||\sigma_{0}(p)|| \leq ||\sigma_{i}(p)|| + \tau || \sum_{i}s_{i} \sigma_{i}(p) || \leq C\tau.
$$
On the other hand, $\overline{\omega} < v^{-1/\beta} \dol b_{0}$ near $F_{0}$.\
Thus, we have
$$
\int_{\cup_{s \in \mathbb{D}^{h(l,m+\beta)}} F_{s}} \overline{\omega}^{n} = O( \tau^2 v^{-n/\beta}) = O( v^{1-n/\beta}).
$$
It follows from the choice of $v>0$ in this theorem that the desired convergence above holds as $(\tau, v , c) \to (0, 0 , \infty)$ by taking sufficiently large $\beta$.\
Then, we obtain a holomorphic vector field
$$
\nabla^{1,0} \phi = g^{i,\overline{j}} \frac{\partial \phi}{\partial \overline{z}^{j}} \frac{\partial}{\partial z^{i}}
$$
on $X \setminus (D \cup F_{0})$.\
Here, we write $\dol (t + \varphi) = \sqrt{-1} g_{i,\overline{j}} dz^{i} \wedge d \overline{z}^{j}$.\
It follows from the definitions of $\phi$ and $g^{i,\overline{j}}$ that $\nabla^{1,0} \phi$ can be extended to $X$.\
The decay condition of $\phi$ and the assumption of holomorphic vector fields on $X$ imply that $\phi = 0$.\
This is contradiction and the resonance theorem (Theorem \ref{resonance}) implies that the inverse operator $ \mathcal{D}_{\overline{\omega}}^{*}\mathcal{D}_{\overline{\omega}}^{-1}$ has an uniform bound.\
\end{proof}

Recall the following relation
$$
L_{\overline{\omega}} = - \mathcal{D}^{*}_{\overline{\omega}}\mathcal{D}_{\overline{\omega}} + (\nabla^{1,0} \ast , \nabla^{0,1} S(\overline{\omega}))_{\overline{\omega}}.
$$
Thus, Theorem \ref{asc and eq} and Theorem \ref{key theorem} imply that Condition \ref{condition A} holds with respect to $\overline{\omega}$.\

\begin{thm}
\label{small small}
Take parameters so that $v^{1/\beta} c = k \log c$ and $\tau^2 = v$.\
Assume that $\theta_{D}$ is cscK and $\mathcal{D}_{\overline{\omega}}^{*}\mathcal{D}_{\overline{\omega}}: C^{4,\alpha}_{\delta - 4} \to C^{0,\alpha}_{\delta}$ is isomorphic.\
Then, we can make the norm of the linear operator $(\nabla^{1,0} \ast, \nabla^{0,1}S(\overline{\omega}))_{\overline{\omega}} = L_{\overline{\omega}} + \mathcal{D}_{\overline{\omega}}^{*}\mathcal{D}_{\overline{\omega}}$ small arbitrarily so that $L_{\overline{\omega}} : C^{4,\alpha}_{\delta - 4} \to C^{0,\alpha}_{\delta}$ is isomorphic.\
Moreover, we can find a constant $\hat{K} > 0$ such that
$$
|| L_{\overline{\omega}} \phi ||_{C^{0,\alpha}_{\delta}} \geq \hat{K} || \phi ||_{C^{4,\alpha}_{\delta - 4}}
$$
for any $c,v,\tau,\phi \in C^{4,\alpha}_{\delta - 4}$.\
\end{thm}

We need the following lemma:
\begin{lemma}
\label{contraction lemma 2}
Assume that $n \geq 5$ and
$$
3 \hat{S}_{D} < n(n-1).
$$
Then, for $\delta > 8$, there exists $c_{0} > 0$ independent of $\overline{\omega}$ such that if $||\phi||_{C^{4,\alpha}_{\delta - 4}(X\setminus D)} \leq c_{0}$, we have
$$
||L_{\overline{\omega}_{\phi}} - L_{\overline{\omega}}||_{C^{4,\alpha}_{\delta - 4} \to C^{0,\alpha}_{\delta}} \leq \hat{K}/2
$$
and $\overline{\omega}_{\phi} = \overline{\omega} + \dol \phi$ is positive.\
\end{lemma}

\begin{proof}
For $\psi \in C^{4,\alpha}_{\delta - 4}$, the following inequality holds:\\

$||(r^{2} + 1)^{\delta/2} (g_{\phi}^{i,\overline{j}}g_{\phi}^{k,\overline{l}} - g^{i,\overline{j}}g^{k,\overline{l}}) \psi_{i,\overline{j},k,\overline{l}}||_{C^{0,\alpha}}$
\begin{eqnarray*}
&\leq& ||(r^{2} + 1)^{4/2} (g_{\phi}^{i,\overline{j}}g_{\phi}^{k,\overline{l}} - g^{i,\overline{j}}g^{k,\overline{l}})||_{C^{0,\alpha}} ||\psi||_{C^{4,\alpha}_{\delta - 4}(X\setminus D)}\\
&=& ||(r^{2} + 1)^{4/2} ( g_{\phi}^{i,\overline{j}}(g_{\phi}^{k,\overline{l}} - g^{k,\overline{l}}) + (g_{\phi}^{i,\overline{j}} - g^{i,\overline{j}})g^{k,\overline{l}})||_{C^{0,\alpha}} ||\psi||_{C^{4,\alpha}_{\delta - 4}(X\setminus D)}.\\
\end{eqnarray*}
In addition,  we have the following equation:
\begin{equation}
\label{inv}
g_{\phi}^{-1} - g^{-1} = g_{\phi}^{-1} ( g - g_{\phi} ) g^{-1}
\end{equation}
for $\phi \in C^{4,\alpha}_{\delta - 4}$ such that $\overline{\omega}_{\phi} = \overline{\omega} + \dol \phi$ is positive.\

It is enough to study the region where $M_{c,v,\eta} = t + \varphi + c$.\
The $C^{2}$-estimate of the degenerate complex {\ma} equation tells us that
\begin{equation}
\label{coefficient}
g^{i,\overline{j}} = O(||\sigma_{D}||^{-2m/l}) = O(r^{2m/l \times n(n-1)/\hat{S}_{D}}).
\end{equation}

Since we have already known the explicit $C^{2,\epsilon}$-estimate of the solution of the degenerate complex {\ma} equation from \cite{Aoi2}, we can estimate the $C^{0,\alpha}$-norm of coefficients $g^{i,\overline{j}},g_{\phi}^{i,\overline{j}}$.\
The hypothesis
$$
\frac{a(n)m}{2l} < \frac{\hat{S}_{D}}{n(n-1)}
$$
implies that $4 + 3 \times 2m/l \times n(n-1)/\hat{S}_{D} - (\delta - 4) < 8 + 12/a(n) - \delta < 0$.\
So, the equation (\ref{inv}) and the estimate (\ref{coefficient}) implies that the term
$$
||(r^{2} + 1)^{4/2} ( g_{\phi}^{i,\overline{j}}(g_{\phi}^{k,\overline{l}} - g^{k,\overline{l}}) + (g_{\phi}^{i,\overline{j}} - g^{i,\overline{j}})g^{k,\overline{l}})||_{C^{0,\alpha}}
$$
is estimated form above by $2c_{0}$.\
By taking a sufficiently small $c_{0}$, we can make the operator norm of $L_{\overline{\omega}_{\phi}} - L_{\overline{\omega}}$ small arbitrarily.\
Thus, we have the desired result.\
\end{proof}

\begin{rem}
The reason why we replace the hypothesis for weights of Banach spaces in the above lemma comes from the $C^{2}$-estimate of the solution of the degenerate complex {\ma} equation due to P$\check{{\rm a}}$un \cite{Pa}.\
From this, the positivity of $\overline{\omega}_{\phi}$ holds.\
On the other hand, we need to assume that $\delta-4 > 4$ to control the factor $(r^{2}+1)^{4/2}$.\
So, the choice of a weight $\delta$ implies that we need to assume that the dimension $n$ is greater than $4$ and $\hat{S}_{D}/n(n-1)$ is smaller than $1/3$.\
In addition, since we need to choose $\epsilon > \delta + 2$ in the proof of Theorem \ref{asc and eq}, we need to assume that $n>5$.\
\end{rem}

Constants $\hat{K}$ and $c_{0}$ which appear in Theorem \ref{small small} and Lemma \ref{contraction lemma 2} respectively, are uniform for parameters $c,v,\tau$.\
Therefore, Theorem \ref{asc and eq} implies that Condition \ref{condition B} holds with respect to $\overline{\omega}$.\

\begin{thm}
\label{small small small}
For the complete {\kah} metric $\overline{\omega}$ above, the inequality
$$
|| S(\overline{\omega}) ||_{C^{0,\alpha}_{\delta}} \leq c_{0} \hat{K} /2
$$
holds by taking suitable parameters $v,c,\tau$.\
\end{thm}

\subsection{The fixed point theorem}

Finally, we show that the existence of a complete scalar-flat {\kah} metric on $X\setminus D$.\
Following Arezzo-Pacard \cite{AP1}, \cite{AP2}, for the expansion of the scalar curvature
$$
S(\overline{\omega} + \dol \phi) = S(\overline{\omega}) + L_{\overline{\omega}} (\phi) + Q_{\overline{\omega}} (\phi),
$$
we consider the following operator
$$
\mathcal{N}(\phi) : = -L_{\overline{\omega}}^{-1}(S(\overline{\omega}) + Q_{\overline{\omega}}(\phi)) \in C^{4,\alpha}_{\delta - 4}
$$
for $\phi \in C^{4,\alpha}_{\delta - 4}$ by following Arezzo-Pacard \cite{AP1}, \cite{AP2} (see also \cite{Sz1}).\
Lemma \ref{contraction lemma 2} implies that $\mathcal{N}$ is the contraction map on the neighborhood of the origin of $C^{4,\alpha}_{\delta - 4}$ for a suitable weight $\delta$.\
In \cite{Aoi1}, we assume that Condition \ref{condition A} and Condition \ref{condition B} hold.\
Namely, we assume that there exists a complete {\kah} metric $\omega_{0}$ whose scalar curvature is sufficiently small so that the operator $L_{\omega_{0}}$ has the uniformly bounded inverse.\
As we have seen, by constructing the {\kah} metric $\overline{\omega}$, Theorem \ref{small small} and Theorem \ref{small small small} imply that we don't have to assume that Condition \ref{condition A} and Condition \ref{condition B} hold.\
The following Proposition implies the existence of a complete scalar-flat {\kah} metric.\

\begin{prop}
\label{key proposition 3}
Set
$$
\mathcal{U} := \left\{ \phi \in C^{4,\alpha}_{\delta - 4} : ||\phi||_{C^{4,\alpha}_{\delta -4}} \leq c_{0} \right\}.
$$
If the assumption in Theorem \ref{complete scalar-flat kah} holds, the operator $\mathcal{N}$ is a contraction on $\mathcal{U}$ and $\mathcal{N}(\mathcal{U}) \subset \mathcal{U}$ by taking suitable parameters $c,v,\tau$.\
\end{prop}

\begin{proof}
Immediately, we have
$$
||\mathcal{N}(\phi)||_{C^{4,\alpha}_{\delta - 4}} \leq ||\mathcal{N}(\phi) - \mathcal{N}(0) ||_{C^{4,\alpha}_{\delta - 4}}  + || \mathcal{N}(0) ||_{C^{4,\alpha}_{\delta - 4}}.
$$
From Lemma \ref{contraction lemma 2} and the condition $||\phi||_{C^{4,\alpha}_{\delta -4}} \leq c_{0}$, the same argument in the proof in \cite{Aoi1} implies that we obtain the following estimate ;
\begin{eqnarray*}
||\mathcal{N}(\phi) - \mathcal{N}(0) ||_{C^{4,\alpha}_{\delta - 4}}
&\leq& || - L_{\overline{\omega}}^{-1} (Q_{\overline{\omega}}(\phi)) ||_{C^{4,\alpha}_{\delta - 4}}\\
&\leq& \hat{K}^{-1} || L_{\overline{\omega} + s \dol \phi} - L_{\overline{\omega}} ||_{C^{4,\alpha}_{\delta - 4} \to C^{0,\alpha}_{\delta}}  ||\phi||_{C^{4,\alpha}_{\delta -4}}
\end{eqnarray*}
for some $s \in [0,1]$.\
Lemma \ref{contraction lemma 2} implies that we have
$$
||\mathcal{N}(\phi) - \mathcal{N}(0) ||_{C^{4,\alpha}_{\delta - 4}}\leq \frac{1}{2} c_{0}
$$
Theorem \ref{small small} and Theorem \ref{small small small} implies that we have
\begin{equation*}
|| \mathcal{N}(0) ||_{C^{4,\alpha}_{\delta - 4}} = ||L_{\overline{\omega}}^{-1} (S(\overline{\omega}))||_{C^{4,\alpha}_{\delta - 4}} \leq \hat{K}^{-1} || S(\overline{\omega}) ||_{C^{0,\alpha}_{\delta}} \leq \frac{1}{2} c_{0}.
\end{equation*}
Thus, $\mathcal{N}(\phi) \in \mathcal{U}$.\
\end{proof}

{
\it Proof of Theorem \ref{complete scalar-flat kah}.}
From the discussion above, there exists a unique $\phi_{\infty} := \lim_{i \to \infty} \mathcal{N}^{i}(\phi)$ for any $\phi \in \mathcal{U} \subset C^{4,\alpha}_{\delta - 4}$ satisfying $\phi_{\infty} = \mathcal{N}(\phi_{\infty})$ under the hypothesis in Theorem \ref{complete scalar-flat kah}.\
Therefore, $\overline{\omega} + \dol \phi_{\infty}$ is a complete scalar-flat {\kah} metric on $X \setminus D$.\

\sq

\bigskip
\address{
(CURRENT ADDRESS)\\
Osaka Prefectural Abuno High School,\\
3-38-1, Himuro-chou, Takatsuki-shi,\\
Osaka, 569-1141\\
Japan
}
{takahiro.aoi.math@gmail.com}

\bigskip
\address{
(OLD ADDRESS)\\
Department of Mathematics\\
Graduate School of Science\\
Osaka University\\
Toyonaka 560-0043\\
Japan
}
{t-aoi@cr.math.sci.osaka-u.ac.jp}

\end{document}